\documentclass[10pt]{amsart}
\usepackage[normalem]{ulem}
\usepackage{amsmath,amssymb,amsthm,graphicx,epstopdf,mathrsfs,url}
\usepackage[usenames,dvipsnames]{color}
\usepackage{dsfont}   
\definecolor{darkred}{rgb}{0.4,0.1,0.1}
\definecolor{darkblue}{rgb}{0.1,0.1,0.4}
\definecolor{darkgrey}{rgb}{0.5,0.5,0.5}
\usepackage[colorlinks=true,linkcolor=darkred,citecolor=Blue]{hyperref}

\numberwithin{equation}{section}
\theoremstyle{plain}
\newtheorem{thm}{Theorem}[section]
\newtheorem{lem}[thm]{Lemma}

\newtheorem{prop}[thm]{Proposition}
\newtheorem{cor}[thm]{Corollary}

\theoremstyle{remark}

\theoremstyle{plain}

\newcommand{\hyp}[1]{$C^{2}$-hypersurface as in Definition~\ref{definition_hypersurface}}

\renewcommand\Im{\textup{Im}\,}

\DeclareMathOperator\ran{ran}

\newcommand{\dom}{\mathrm{dom}\,}

\begin{document}
\title[]{On Dirac operators in \boldmath{$\mathbb{R}^3$} with electrostatic and Lorentz scalar \boldmath{$\delta$}-shell interactions}

\author[J. Behrndt]{Jussi Behrndt}
\address{Institut f\"{u}r Angewandte Mathematik\\
Technische Universit\"{a}t Graz\\
 Steyrergasse 30, 8010 Graz, Austria\\
E-mail: {behrndt@tugraz.at}}

\author[P. Exner]{Pavel Exner}
\address{Doppler Institute for Mathematical Physics and Applied Mathematics\\ 
Czech Technical University in Prague\\ B\v{r}ehov\'{a} 7, 11519 Prague, Czech Republic,
{\rm and}
Department of Theoretical Physics\\ 
Nuclear Physics Institute, 
Czech Academy of Sciences, 
25068 \v{R}e\v{z}, Czech Republic\\
E-mail: {exner@ujf.cas.cz}
}

\author[M. Holzmann]{Markus Holzmann}
\address{Institut f\"{u}r Angewandte Mathematik\\
Technische Universit\"{a}t Graz\\
 Steyrergasse 30, 8010 Graz, Austria\\
E-mail: {holzmann@math.tugraz.at}}

\author[V. Lotoreichik]{Vladimir Lotoreichik}
\address{
Department of Theoretical Physics\\
Nuclear Physics Institute, Czech Academy of Sciences, 
25068 \v{R}e\v{z}, Czech Republic\\
E-mail: {lotoreichik@ujf.cas.cz }
}

\begin{abstract}
  In this article Dirac operators $A_{\eta, \tau}$ coupled with combinations of electrostatic and Lorentz scalar $\delta$-shell interactions of constant strength $\eta$ and $\tau$, respectively, supported on compact surfaces $\Sigma \subset \mathbb{R}^3$ are studied. In the rigorous definition of these operators the $\delta$-potentials are modelled by coupling conditions at $\Sigma$. In the proof of the self-adjointness of $A_{\eta, \tau}$ a Krein-type resolvent formula and a Birman-Schwinger principle are obtained. With their help a detailed study of the qualitative spectral properties of $A_{\eta, \tau}$ is possible. In particular, the essential spectrum of $A_{\eta, \tau}$ is determined, it is shown that at most finitely many discrete eigenvalues can appear, and several symmetry relations in the point spectrum are obtained. Moreover, the nonrelativistic limit of $A_{\eta, \tau}$ is computed and it is discussed that for some special interaction strengths $A_{\eta, \tau}$ is decoupled to two operators acting in the domains with the common boundary $\Sigma$.
\end{abstract}


\maketitle

\section{Introduction}

Working with the equations of motion there is a particular interest to find solutions which are exact and which correspond to specific physical systems. Such an ideal treatment was possible, for instance, in the quantum mechanical explanation of the spectral properties of one-electron atoms. However, such situations are rare and hence, the original model is often replaced by an idealized one which is mathematically accessible and reflects at the same time the physical reality to a reasonable degree. In many problems this can be achieved by using singular potentials supported on sets of measure zero. This method is used highly successfully in nonrelativistic quantum mechanics, cf. the monograph~\cite{AGHH05}.

Life becomes more complicated when the systems under consideration are relativistic, described by the Dirac equation. Here there are only very few solvable models and the physics becomes more complicated when other than electromagnetic forces enter the picture. An example of such a situation is the quark dynamics within the nucleon. An early attempt to describe it was made by Bogolioubov, Struminski and Tavkhelidze, as cited in~\cite{Bog68}, who proposed to model them as confined to a spherical cavity. The nature of the confinement was not consistent there, but the idea inspired a little later the so-called MIT bag model~\cite{C75, CJJT74, CJJTW74, DJJK75, J75}. 

The requirement of relativistic invariance allows to distinguish several types of potentials specified by their behaviour with respect to the Lorentz group transformations~\cite[Section 4.2]{T92}. The most common among them are the scalar and electromagnetic ones, and among the latter the electrostatic one plays an important role. In this paper, we consider combinations of scalar and electrostatic potentials, which differ by the presence and absence, respectively, of the Dirac matrix $\beta$; a useful feature of such potential combinations is that the mentioned matrix gives rise to one of the possible supersymmetries of the Dirac equation~\cite[Section~5.1]{T92}. 

Let us now describe the aim of the paper in more detail.
To set the stage let $\Sigma \subset \mathbb{R}^3$ be a closed, bounded, and sufficiently smooth surface which splits~$\mathbb{R}^3$ into a bounded domain $\Omega_+$ and an unbounded domain $\Omega_-$, and let $\nu$ be the unit normal vector field at $\Sigma$ pointing outwards $\Omega_+$.
Our goal is to study Dirac operators acting in $L^2(\mathbb{R}^3)^4$ which are formally given by
\begin{equation} \label{def_A_eta_formal}
  A_{\eta, \tau} = -i c \sum_{j=1}^3 \alpha_j \partial_j + m c^2 \beta + (\eta I_4 + \tau \beta) \delta_\Sigma,
\end{equation}
where $m$ is the mass of the particle, $c$ is the speed of light, $\alpha_1, \alpha_2, \alpha_3, \beta \in \mathbb{C}^{4 \times 4}$ are the Dirac matrices defined in \eqref{def_Dirac_matrices} below, $I_4 \in \mathbb{C}^{4 \times 4}$ is the identity matrix, $\eta, \tau \in \mathbb{R}$ are the interaction strengths, and the $\delta$-distribution acts in a symmetric way as
\begin{equation*}
  \delta_\Sigma f = \frac{1}{2} (f_+|_\Sigma + f_-|_\Sigma), \qquad f_\pm = f \upharpoonright \Omega_\pm.
\end{equation*}
In order to introduce $A_{\eta, \tau}$ in a mathematically rigorous form as a self-adjoint operator in $L^2(\mathbb{R}^3)^4$ we require that functions in the domain of $A_{\eta, \tau}$ satisfy suitable coupling conditions on $\Sigma$. To find them, we note first that the distribution $A_{\eta, \tau} f$ acts on a test function $\varphi \in C_0^\infty(\mathbb{R}^3)^4$ as
\begin{equation*}
  \langle A_{\eta, \tau} f, \varphi \rangle = \int_{\mathbb{R}^3} f \cdot \overline{\left( -i c \alpha \cdot \nabla \varphi + m c^2 \beta \varphi \right)} \text{d} x + \int_\Sigma (f_+|_\Sigma + f_-|_\Sigma) \cdot \frac{1}{2} (\eta I_4 + \tau \beta) \overline{\varphi|_\Sigma} \text{d} \sigma,
\end{equation*}
where the notation $\alpha \cdot x = \alpha_1 x_1 + \alpha_2 x_2 + \alpha_3 x_3$ for a vector $x = (x_1,x_2,x_3)$ was used.
On the other hand, one would expect that the operator $A_{\eta, \tau}$ acts for $x \notin \Sigma$ as 
\begin{equation} \label{action_A_eta_formal}
  A_{\eta, \tau} f(x) = -i c \alpha \cdot \nabla f(x) + m c^2 \beta f(x),
\end{equation}
which leads via integration by parts in $\Omega_\pm$ to the observation that
\begin{equation*}
  \begin{split}
    \langle A_{\eta, \tau} f, \varphi \rangle &= \int_{\Omega_+ \cup \Omega_-} \big( -i c \alpha \cdot \nabla f + m c^2 \beta f \big) \cdot \overline{\varphi} \text{d} x\\
    &=\int_{\mathbb{R}^3} f \cdot \overline{\left( -i c \alpha \cdot \nabla \varphi + m c^2 \beta \varphi \right)} \text{d} x - \int_\Sigma i c \alpha \cdot \nu (f_+|_\Sigma - f_-|_\Sigma) \cdot \overline{\varphi|_\Sigma} \text{d} \sigma
  \end{split}
\end{equation*}
should hold for $f \in \dom A_{\eta, \tau}$. Comparing the two expressions for $\langle A_{\eta, \tau} f, \varphi \rangle$ we conclude that a function $f \in \dom A_{\eta, \tau}$ should satisfy the jump condition
\begin{equation} \label{equation_jump_condition_intro}
  -i c \alpha \cdot \nu (f_+|_\Sigma - f_-|_\Sigma) = \frac{1}{2} (\eta I_4 + \tau \beta) (f_+|_\Sigma + f_-|_\Sigma).
\end{equation}
Therefore, the operator $A_{\eta, \tau}$ corresponding to the formal differential expression~\eqref{def_A_eta_formal} should be defined for functions satisfying the coupling condition~\eqref{equation_jump_condition_intro} and should act for $x \notin \Sigma$ as in~\eqref{action_A_eta_formal}.

The mathematical study of Dirac operators with singular potentials started in the 1980s, when Gesztesy and \v{S}eba considered one dimensional Dirac operators with point interactions \cite{AGHH05, GS87, S89}; for more recent contributions on Dirac operators with point interactions see, e.g., \cite{BMP17, CMP13, PR14}. Based on \cite{GS87} and a decomposition to spherical harmonics Dittrich, Exner, and \v{S}eba investigated the operator $A_{\eta, \tau}$ in the case that $\Sigma$ is the sphere in $\mathbb{R}^3$. In \cite{DES89} they showed for a wide class of parameters the self-adjointness of $A_{\eta, \tau}$ and they were able to compute its resolvent and some of its spectral properties. While some of the interesting properties of $A_{\eta, \tau}$ like the decoupling of the operator to two Dirac operators acting in $\Omega_\pm$ for interaction strengths satisfying $\eta^2 - \tau^2 = -4 c^2$ were observed in \cite{DES89}, compare also Lemma~\ref{lemma_transmission_condition} below, others like, e.g., unexpected spectral effects for $\eta^2 - \tau^2 = 4 c^2$ could not be seen with this approach due to the decomposition to the spherical harmonics.

It took then 25 years until Dirac operators with singular interactions supported on more general surfaces in $\mathbb{R}^3$ were studied. In a series of papers \cite{AMV14, AMV15, AMV16} Arrizabalaga, Mas, and Vega showed the self-adjointness and derived several basic properties of $A_{\eta, \tau}$, in particular for the special case of purely electrostatic interactions, i.e. for $\tau = 0$. Moreover, for purely electrostatic and purely scalar interactions it was shown in \cite{MP18} that $A_{\eta, \tau}$ can be regarded as a limit of Dirac operators with squeezed potentials. Inspired by the approach in \cite{AMV14} the authors of the present paper applied the abstract concept of quasi boundary triples and Weyl functions from extension theory of symmetric operators to Dirac operators with singular interactions and provided in the recent paper  \cite{BEHL17} a deeper analysis of the spectral properties of $A_{\eta, 0}$ for purely electrostatic potentials. We should note that in all of the above mentioned papers the case $\eta^2-\tau^2=4c^2$ was excluded and it turns out that in this {\it critical case} the operator $A_{\eta, \tau}$ has different properties as in the {\it noncritical case} $\eta^2 - \tau^2 \neq 4 c^2$. For purely electrostatic interactions the self-adjointness of $A_{\eta, 0}$ for critical $\eta = \pm 2 c$ was studied in \cite{BH17, OV17} and some surprising spectral effects like possible appearance of additional essential spectrum were shown. Eventually, in \cite{HOP17} a detailed study of the spectral properties of $A_{0, \tau}$ for purely scalar potentials was provided; in particular, it was shown that the discrete eigenvalues in the large mass limit are characterized by an effective operator on the surface $\Sigma$. Furthermore, there is a great interest recently in the study of self-adjoint Dirac operators on domains with boundary conditions, see, e.g., \cite{ALTR17, ALTMR18, BFSB17_1, BFSB17_2, FS14, LTO18, LO18, MOP18, S95}.

Our goal in this note is to extend many of the above mentioned results, which were shown for purely electrostatic or purely scalar interactions, to the more general case of combinations of electrostatic and scalar interactions. For that we use a uniform approach which is based on the considerations in \cite{BEHL17, BH17}. After presenting some preliminary material on integral operators which are associated to the Green function of the resolvent of the free Dirac operator, we introduce in Section~\ref{section_self_adjoint} the operator $A_{\eta, \tau}$ in a mathematically rigorous way via the coupling condition~\eqref{equation_jump_condition_intro}. Then we show for noncritical interaction strengths $\eta^2 - \tau^2 \neq 4 c^2$ the self-adjointness of $A_{\eta, \tau}$ in Theorem~\ref{theorem_self_adjoint}. In the proof of the self-adjointness  we also verify a Birman-Schwinger principle, which translates the eigenvalue problem for the differential operator $A_{\eta, \tau}$ to a nonlinear eigenvalue problem for a family of integral operators acting on $\Sigma$.

In Section~\ref{section_spectral_properties} we provide the basic spectral properties of $A_{\eta, \tau}$ for noncritical interaction strengths. We compute the essential spectrum, show that at most finitely many discrete eigenvalues appear and obtain several symmetry relations for the spectrum of $A_{\eta, \tau}$. We complement the results for noncritical interactions by a theorem from \cite{BH17} which shows that the spectral properties of $A_{\eta, \tau}$ can be completely different in the critical case. 

Finally, we compute in Section~\ref{section_nonrelativistic_limit_singular_interaction} for purely electrostatic and purely scalar potentials the nonrelativistic limit of $A_{\eta, \tau}$, which shows that $A_{\eta, \tau}$ is the relativistic counterpart of the Schr\"odinger operator $-\frac{1}{2m} \Delta + \eta \delta_\Sigma$ and which gives another justification that the jump condition~\eqref{equation_jump_condition_intro} models the $\delta$-potential correctly.

\subsection*{Acknowledgement}
The authors acknowledge financial support under the Czech-Austrian grant 7AMB17AT022 and CZ 02/2017. 
PE and VL are supported by the Czech Science Foundation (GA\v{C}R), Grant No. 17-01706S. PE also
acknowledges the support by the European Union within the project CZ.02.1.01/0.0/0.0/16 019/0000778.

\section{The free Dirac operator and associated integral operators} \label{section_free_Operator}

In this preliminary section we collect some well known facts about the free Dirac operator in $\mathbb{R}^3$ and some associated integral operators that are needed to investigate Dirac operators with singular $\delta$-shell interactions. For that we have to fix some notations first. 

\subsection{Notations}

Let $\mathcal{H}$ be a Hilbert space. Then we write $\mathcal{H}^d := \mathcal{H} \otimes \mathbb{C}^d$. For a closable operator $A$ in $\mathcal{H}$ its domain of definition, its range, and its kernel are denoted by $\dom A$, $\ran A$, and $\ker T$, respectively. The closure of $A$ is $\overline{A}$. Eventually, if $A$ is self-adjoint, then its resolvent set, its spectrum, the point, discrete, and essential spectrum are $\rho(A)$, $\sigma(A)$, $\sigma_\text{p}(A)$, $\sigma_\text{disc}(A)$, and $\sigma_\text{ess}(A)$, respectively.

For a domain $\Omega \subset \mathbb{R}^3$ with a compact $C^2$-smooth boundary $\Sigma := \partial \Omega$ we denote by $L^2(\Omega)$ the standard $L^2$-spaces and $L^2(\Sigma)$ is endowed with the inner product based on the integral with respect to the surface measure $\sigma$. As usual, $H^1(\Omega)$ stands for the Sobolev space of order one which consists of functions $f \in L^2(\Omega)$ with $\nabla f \in L^2(\Omega)^3$, where $\nabla f$ is the distributional gradient of~$f$. Similarly $H^1(\mathbb{R}^3)$ is introduced. Moreover, we define the trace space 
\begin{equation*}
  H^{1/2}(\Sigma) := \{ f |_\Sigma: f \in H^1(\Omega) \}
\end{equation*}
equipped with the norm $\| \varphi \|_{1/2} := \inf\{ \| f \|_{H^1(\Omega)}: f \in H^1(\Omega),~ f|_\Sigma = \varphi \}$. One verifies that the trace mapping 
\begin{equation} \label{trace_theorem}
  H^1(\Omega) \ni f \mapsto f|_\Sigma \in H^{1/2}(\Sigma)
\end{equation}
is a bounded, surjective linear map and one can further show that $H^{1/2}(\Sigma) \subset L^2(\Sigma)$, cf. \cite[Section~4.2 and Theorem~4.2.1]{HW08}.

Since we are not
interested in the semiclassical limit, we choose units in~\eqref{def_A_eta_formal} in
such a way that $\hbar=1$. However, we keep the mass of the particle $m$
and the speed of light $c$ both as positive constants. The Dirac matrices $\alpha := (\alpha_1, \alpha_2, \alpha_3)$ and $\beta$ are defined for $j \in \{ 1,2,3\}$ by
\begin{equation} \label{def_Dirac_matrices}
  \alpha_j := \begin{pmatrix} 0 & \sigma_j \\ \sigma_j & 0 \end{pmatrix}
  \quad \text{and} \quad \beta := \begin{pmatrix} I_2 & 0 \\ 0 & -I_2 \end{pmatrix},
\end{equation}
where $I_d$ denotes the $d\times d$-identity matrix and $\sigma_1, \sigma_2, \sigma_3$ are the Pauli spin matrices 
\begin{equation*}
  \sigma_1 := \begin{pmatrix} 0 & 1 \\ 1 & 0 \end{pmatrix}, \qquad
  \sigma_2 := \begin{pmatrix} 0 & -i \\ i & 0 \end{pmatrix}, \qquad
  \sigma_3 := \begin{pmatrix} 1 & 0 \\ 0 & -1 \end{pmatrix}.
\end{equation*}
It is easy to see that the Dirac matrices satisfy
\begin{equation} \label{equation_anti_commutation}
  \alpha_j \alpha_k + \alpha_k \alpha_j = 2 \delta_{j k} I_4 \quad \text{and} \quad \alpha_j \beta + \beta \alpha_j = 0, \qquad j, k \in \{ 1, 2, 3\}.
\end{equation}
For $x = (x_1, x_2, x_3) \in \mathbb{R}^3$ we will often employ the notations
\begin{equation*}
  \alpha \cdot x = \sum_{k=1}^3 \alpha_k x_k \quad \text{and} \quad \alpha \cdot \nabla = \sum_{k=1}^3 \alpha_k \partial_k.
\end{equation*}

Finally, if not stated differently, $\Omega_+ \subset \mathbb{R}^3$ is always a bounded domain with compact $C^2$-smooth boundary $\Sigma$, $\Omega_- = \mathbb{R}^3 \setminus \overline{\Omega_+}$, and $\nu$ denotes the unit normal vector field at $\Sigma$ pointing outwards $\Omega_+$. We will often write $f_\pm := f \upharpoonright \Omega_\pm$ for $f \in L^2(\mathbb{R}^3)$.

\subsection{The free Dirac operator}

We are now prepared to introduce the free Dirac operator, which acts in the Hilbert space $L^2(\mathbb{R}^3)^4$ as
\begin{equation} \label{def_free_Dirac}
    A_0 f := -i c \sum_{j=1}^3 \alpha_j \partial_j f + m c^2 \beta f, 
    \qquad
   \dom A_0 := H^1(\mathbb{R}^3)^4.
\end{equation}
Using the Foldy-Wouthuysen transformation it is easy to see that $A_0$ is self-adjoint and that
\begin{equation*}
  \sigma(A_0) = \sigma_{\textup{ess}}(A_0) = (-\infty, -m c^2] \cup [m c^2, \infty),
\end{equation*}
cf. \cite[Section~1.4]{T92}.
Next, for $\lambda \in \rho(A_0) = \mathbb{C} \setminus \big( (-\infty, -m c^2] \cup [m c^2, \infty) \big)$
the resolvent of $A_0$ is
\begin{equation*}
  (A_0 - \lambda)^{-1} f(x) = \int_{\mathbb{R}^3} G_\lambda(x-y) f(y) \textup{d} y, \quad f \in L^2(\mathbb{R}^3)^4,~x \in \mathbb{R}^3,
\end{equation*}
where the Green function $G_\lambda$ is given for $x\neq0$ by
\begin{equation*} 
 \begin{split}
   G_\lambda(x) = \left( \frac{\lambda}{c^2} I_4 + m \beta 
                   + \left( 1 - i \sqrt{\frac{\lambda^2}{c^2} - (m c)^2} |x| \right) \frac{i}{c |x|^2} \alpha \cdot x \right)
                   \cdot \frac{e^{i \sqrt{\lambda^2/c^2 - (m c)^2} |x|}}{4 \pi |x|}&,
 \end{split}
\end{equation*}
see \cite[Section~1.E]{T92};
in the last formula the convention $\Im \sqrt{\lambda^2/c^2 - (m c)^2} > 0$ is used.

\subsection{Auxiliary integral operators}

In this subsection we introduce several families of integral operators which are related to the Green function~$G_\lambda$ and which will play a crucial role later in the study of Dirac operators with singular $\delta$-shell interactions.
For a fixed~$\lambda \in \rho(A_0) = \mathbb{C} \setminus \big( (-\infty, -m c^2] \cup [m c^2, \infty) \big)$
we define the potential operator $\Phi_\lambda: L^2(\Sigma)^4 \rightarrow L^2(\mathbb{R}^3)^4$ by
\begin{equation} \label{def_Phi_lambda}
  \Phi_\lambda \varphi(x) := \int_\Sigma G_\lambda(x-y) \varphi(y) \textup{d}\sigma(y), 
  \quad \varphi \in L^2(\Sigma)^4,~x \in \mathbb{R}^3,
\end{equation}
and the strongly singular boundary integral operator $\mathcal{C}_\lambda: L^2(\Sigma)^4 \rightarrow L^2(\Sigma)^4$ acting as
\begin{equation} \label{def_C_lambda}
  \mathcal{C}_\lambda \varphi(x) := \lim_{\varepsilon \searrow 0} \int_{\Sigma \setminus B(x, \varepsilon)} 
  G_\lambda(x-y) \varphi(y) \textup{d}\sigma(y), \quad
  \varphi \in L^2(\Sigma)^4,~x \in \Sigma,
\end{equation}
where $B(x,\varepsilon)$ is the ball of radius $\varepsilon$ centered at $x$.
Both operators $\Phi_\lambda$ and $\mathcal{C}_\lambda$ are well defined and bounded, 
see~\cite[Proposition~3.4]{BEHL17} or \cite[Section~2]{AMV15}, 
and $\Phi_\lambda$ is injective by~\cite[Proposition~3.4 and Definition~2.3]{BEHL17}. In particular, $\mathcal{C}_\lambda$ is uniformly bounded for $\lambda \in (-m c^2, m c^2)$, i.e. there exists a constant $K > 0$ independent of $\lambda$ such that
\begin{equation} \label{C_lambda_uniformly_bounded}
  \| \mathcal{C}_\lambda \| \leq K \qquad \text{ for all } \lambda \in (-m c^2,m c^2),
\end{equation}
cf. \cite[Proposition~3.5]{BEHL17} and also \cite[Lemma~3.2]{AMV15}. 
Next, if $\varphi \in H^{1/2}(\Sigma)^4$, then according to~\cite[Proposition~4.2]{BH17} 
\begin{equation} \label{equation_Phi_smooth}
  \Phi_\lambda \varphi \in H^1(\Omega_+)^4 \oplus H^1(\Omega_-)^4 \quad \text{and} \quad 
  \mathcal{C}_\lambda \varphi \in H^{1/2}(\Sigma)^4
\end{equation}
hold. Moreover, if  $\lambda \in \rho(A_0)$, then a function $f_\lambda \in H^1(\Omega_+)^4 \oplus H^1(\Omega_-)^4$ satisfies 
\begin{equation*}
  (-i c \alpha \cdot \nabla + m c^2 \beta - \lambda) f_\lambda = 0 \text{ in } \Omega_\pm,
\end{equation*}
if and only if there exists a density $\varphi \in H^{1/2}(\Sigma)^4$ such that 
\begin{equation} \label{equation_kernel}
  f_\lambda = \Phi_\lambda \varphi;
\end{equation}
see~\cite[Proposition~4.2]{BH17}. 

Now, we describe how $\Phi_\lambda$ and $\mathcal{C}_\lambda$ are related to each other by taking traces. 
Let $\varphi \in H^{1/2}(\Sigma)^4$ and $\lambda \in \rho(A_0)$. Then, the trace of the function
\begin{equation*}
  \Phi_\lambda \varphi = (\Phi_\lambda \varphi)_+ \oplus (\Phi_\lambda \varphi)_-  
  \in H^1(\Omega_+)^4 \oplus H^1(\Omega_-)^4
\end{equation*}
on $\Sigma$ is 
\begin{equation*}
 \big((\Phi_\lambda \varphi)_\pm\big) \big|_\Sigma = \mathcal{C}_\lambda \varphi \mp \dfrac{i}{2 c} (\alpha \cdot \nu) \varphi;
\end{equation*}
this is shown in~\cite[Lemma~2.2]{AMV15} for $\lambda \in (-m c^2, m c^2)$, the case $\lambda \in \mathbb{C} \setminus \mathbb{R}$
can be proved in the same way. In particular, using $(\alpha \cdot \nu)^2 = I_4$ one finds that the identities
\begin{gather} \label{jump1}
  \frac{1}{2} \big( (\Phi_\lambda \varphi)_+|_\Sigma + (\Phi_\lambda \varphi)_-|_\Sigma \big) = \mathcal{C}_\lambda \varphi,\\
 \label{jump2}
  i c \alpha \cdot \nu \big( (\Phi_\lambda \varphi)_+|_\Sigma - (\Phi_\lambda \varphi)_-|_\Sigma \big) = \varphi,
\end{gather}
hold. Finally, let us mention the mapping properties of the operators $\mathcal{C}_\lambda^2 - \frac{1}{4 c^2} I_4$ and $\mathcal{C}_\lambda \beta + \beta \mathcal{C}_\lambda$ which will be important for the analysis of $A_{\eta, \tau}$. Using the anti-commutation relation~\eqref{equation_anti_commutation} it is easy to see for $\varphi \in L^2(\Sigma)^4$ that 
\begin{equation*}
  (\beta \mathcal{C}_\lambda + \mathcal{C}_\lambda \beta) \varphi(x) =  2 \left(\frac{\lambda}{c^2} \beta + m I_4 \right)
                   \cdot \int_{\Sigma} \frac{e^{i \sqrt{\lambda^2/c^2 - (m c)^2} |x-y|}}{4 \pi |x-y|} \varphi(y) \text{d} \sigma(y),
\end{equation*}
i.e. $\beta \mathcal{C}_\lambda + \mathcal{C}_\lambda \beta$ is a constant matrix times the single-layer boundary integral operator associated to $-\Delta + (m c)^2 - \frac{\lambda^2}{c^2}$, cf. \cite[equation~(9.15)]{M00}. This together with \cite[Theorem~6.11]{M00}, the fact that $H^{1/2}(\Sigma)^4$ is compactly embedded in $L^2(\Sigma)^4$, see, e.g., \cite[Theorem~4.2.2]{HW08}, and \cite[Proposition~4.4~(iii)]{BH17}, see also~\cite[Proposition~2.8]{OV17}, yields the following proposition:

\begin{prop} \label{proposition_mapping_properties}
  Let $\lambda \in \rho(A_0) = \mathbb{C} \setminus \big( (-\infty, - m c^2] \cup [m c^2, \infty) \big)$. Then the following holds.
  \begin{itemize}
    \item[\textup{(i)}] The operator $\mathcal{C}_\lambda^2 - \frac{1}{4 c^2} I_4$ gives rise to a bounded operator
    \begin{equation*}
      \mathcal{C}_\lambda^2 - \frac{1}{4 c^2} I_4: L^2(\Sigma)^4 \rightarrow H^{1/2}(\Sigma)^4.
    \end{equation*}
    In particular, $\mathcal{C}_\lambda^2 - \frac{1}{4 c^2} I_4$ is compact in $L^2(\Sigma)^4$.
    \item[\textup{(ii)}] The operator $\beta \mathcal{C}_\lambda + \mathcal{C}_\lambda \beta$ gives rise to a bounded operator
    \begin{equation*}
      \beta \mathcal{C}_\lambda + \mathcal{C}_\lambda \beta: L^2(\Sigma)^4 \rightarrow H^{1/2}(\Sigma)^4.
    \end{equation*}
    In particular, $\beta \mathcal{C}_\lambda + \mathcal{C}_\lambda \beta$ is compact in $L^2(\Sigma)^4$.
  \end{itemize}
\end{prop}

Finally, we note that the adjoint $\Phi_\lambda^*: L^2(\mathbb{R}^3)^4 \rightarrow L^2(\Sigma)^4$ of $\Phi_\lambda$
acts as 
\begin{equation} \label{equation_Phi_lambda_star}
  \Phi_\lambda^* f = \big( (A_0 - \overline{\lambda})^{-1} f \big)\big|_\Sigma
\end{equation}
or, in a more explicit way,
\begin{equation*}
  \Phi_\lambda^* f(x) = \int_{\mathbb{R}^3} G_{\bar{\lambda}}(x-y) f(y) \textup{d} y, \quad
  f \in L^2(\mathbb{R}^3)^4,~x \in \Sigma.
\end{equation*}
It follows from \eqref{trace_theorem}, \eqref{def_free_Dirac}, and \eqref{equation_Phi_lambda_star} that $\Phi_\lambda^* f \in H^{1/2}(\Sigma)^4$ for any $f \in L^2(\mathbb{R}^3)^4$.

\section{Definition and self-adjointness of $A_{\eta, \tau}$} \label{section_self_adjoint}

This section is devoted to the rigorous mathematical definition of the operator $A_{\eta, \tau}$ and the proof of its self-adjointness. 
In the following we will often make use of the orthogonal decomposition 
$L^2(\mathbb{R}^3)^4 = L^2(\Omega_+)^4 \oplus L^2(\Omega_-)^4$
and we write for $f \in L^2(\mathbb{R}^3)^4$, in this sense, $f = f_+ \oplus f_-$ with 
$f_\pm := f \upharpoonright \Omega_\pm$.

As explained in the introduction, see~\eqref{equation_jump_condition_intro}, the $\delta$-shell interaction is modeled by a coupling condition which has to be satisfied by functions in the operator domain. 
We define for $\eta, \tau \in \mathbb{R}$ the operator $A_{\eta, \tau}$ by
\begin{equation} \label{def_A_eta}
    \begin{split}
      A_{\eta, \tau} f 
        &:= (-i c \alpha \cdot \nabla + m c^2 \beta) f_+ \oplus (-i c \alpha \cdot \nabla + m c^2 \beta) f_-, \\
      \dom A_{\eta, \tau} 
          &:= \big\{ f = f_+ \oplus f_- \in H^1(\Omega_+)^4 \oplus H^1(\Omega_-)^4: \\
          & \qquad \qquad i c \alpha \cdot \nu (f_+|_\Sigma - f_-|_\Sigma) + \tfrac{1}{2}(\eta I_4 + \tau \beta) ( f_+|_\Sigma + f_-|_\Sigma) = 0 \big\}.
    \end{split}
  \end{equation}

In the following lemma we discuss some alternative representations of the coupling condition which models the $\delta$-shell interaction:

\begin{lem} \label{lemma_transmission_condition}
  Let $\eta, \tau \in \mathbb{R}$. Then the following hold.
  \begin{itemize}
    \item[\textup{(i)}] If $\eta^2 - \tau^2 \neq -4 c^2$, then there exists an invertible matrix $R_{\eta, \tau}$ given explicitly in~\eqref{def_R_eta} such that a function $f = f_+ \oplus f_- \in H^1(\Omega_+)^4 \oplus H^1(\Omega_-)^4$ belongs to $\dom A_{\eta, \tau}$ if and only if
    \begin{equation*} 
      f_+|_\Sigma = R_{\eta, \tau} f_-|_\Sigma.
    \end{equation*}
    \item[\textup{(ii)}] If $\eta^2 - \tau^2 = -4 c^2$, then a function $f = f_+ \oplus f_- \in H^1(\Omega_+)^4 \oplus H^1(\Omega_-)^4$ belongs to $\dom A_{\eta, \tau}$ if and only if
    \begin{equation*}
      \big( 2 c I_4 - i(\alpha \cdot \nu) (\eta I_4 + \tau \beta) \big) f_+|_\Sigma = 0, \quad 
      \big( 2 c I_4 + i(\alpha \cdot \nu)  (\eta I_4 + \tau \beta) \big) f_-|_\Sigma = 0.
    \end{equation*}
  \end{itemize}
\end{lem}

Before we prove Lemma~\ref{lemma_transmission_condition} let us discuss its meaning: if $\eta^2 - \tau^2 \neq -4 c^2$, then item~(i) shows that \eqref{equation_jump_condition_intro} is a coupling condition which relates the values of $f_+$ at $\Sigma$ to those of $f_-$ at $\Sigma$ via the matrix $R_{\eta, \tau}$. 
On the other hand, if $\eta^2 - \tau^2 = -4 c^2$, then assertion~(ii) of the above lemma shows that $A_{\eta, \tau}$ is decoupled to Dirac operators in $\Omega_\pm$ with the above boundary conditions. This implies a confinement meaning that a particle which is initially located in $\Omega_\pm$ will remain in $\Omega_\pm$ in its time evolution. In other words this means that the $\delta$-potential makes $\Sigma$ impenetrable for particles. This is investigated in a more detailed way in \cite[Section~5]{AMV15} and \cite[Section~V]{DES89}. In particular, using the anti-commutation relation \eqref{equation_anti_commutation} we see that the above boundary conditions simplify for $\eta = 0$ and $\tau=2c$ to
\begin{equation*}
  \big( I_4 + i \beta (\alpha \cdot \nu) \big) f_+|_\Sigma = 0, \quad 
  \big( I_4 - i \beta (\alpha \cdot \nu) \big) f_-|_\Sigma = 0,
\end{equation*}
which are the boundary conditions characterizing the MIT bag model of quarks confined in a nucleon mentioned in the introduction \cite{C75, CJJT74, CJJTW74, DJJK75, J75} (note that the normal $\nu$ is pointing inside $\Omega_-$). In this way, $A_{0, 2c}$ decomposes into the orthogonal sum of an MIT bag operator in $\Omega_+$ and a Dirac operator in the ``exterior bag'' $\Omega_-$ with similar boundary conditions. We remark that from the physical point of view
only the problem on a bounded domain is a model for the quark confinement, while its direct counterpart on an exterior unbounded domain is merely a mathematical object.

\begin{proof}[Proof of Lemma~\ref{lemma_transmission_condition}]
Another way to write the coupling condition~\eqref{equation_jump_condition_intro} is
\begin{equation} \label{equation_transmission_condition1}
  \left( i c (\alpha \cdot \nu) + \frac{1}{2} (\eta I_4 + \tau \beta) \right) f_+|_\Sigma
  + \left( -i c (\alpha \cdot \nu) + \frac{1}{2} (\eta I_4 + \tau \beta) \right) f_-|_\Sigma = 0.
\end{equation}
If $\eta^2 - \tau^2 \neq -4 c^2$, then the matrix
$i c (\alpha \cdot \nu) + \frac{1}{2} (\eta I_4 + \tau \beta)$ is invertible with 
\begin{equation*}
  \left(i c (\alpha \cdot \nu) + \frac{1}{2} (\eta I_4 + \tau \beta)\right)^{-1} = \frac{4}{4 c^2 + \eta^2 - \tau^2} \left(- i c (\alpha \cdot \nu) + \frac{1}{2} (\eta I_4 - \tau \beta)\right).
\end{equation*}
Hence, if we set 
\begin{equation} \label{def_R_eta}
  R_{\eta, \tau} := -\left( i c (\alpha \cdot \nu) + \frac{1}{2} (\eta I_4 + \tau \beta) \right)^{-1}
      \left( - i c (\alpha \cdot \nu) + \frac{1}{2} (\eta I_4 + \tau \beta) \right),
\end{equation}
then we deduce immediately the result of item~(i). To show assertion~(ii) one just has to multiply~\eqref{equation_transmission_condition1} by the matrices $\pm i c (\alpha \cdot \nu) + \frac{1}{2} (\eta I_4 - \tau \beta)$. Using~\eqref{equation_anti_commutation} and $\eta^2 - \tau^2 = -4 c^2$ one finds that these equations simplify to the claimed boundary conditions.
\end{proof}

Using integration by parts and the coupling condition~\eqref{equation_jump_condition_intro} we show first that $A_{\eta, \tau}$ is symmetric:

\begin{lem} \label{lemma_symmetric}
  Let $\eta, \tau \in \mathbb{R}$. Then the operator $A_{\eta, \tau}$ defined by~\eqref{def_A_eta}  is symmetric.
\end{lem}
\begin{proof}
  Let $f, g  \in \dom A_{\eta, \tau}$. Employing integration by parts in $\Omega_\pm$ we get first
  \begin{equation*}
    \begin{split}
      (&A_{\eta, \tau} f, g )_{L^2(\mathbb{R}^3)^4}
              - (f, A_{\eta, \tau} g)_{L^2(\mathbb{R}^3)^4} \\
      &= (-i c \alpha \cdot \nu f_+, g_+ )_{L^2(\Sigma)^4}   -( -i c \alpha \cdot \nu f_-, g_- )_{L^2(\Sigma)^4} \\
      &= \frac{1}{2} \big( -i c \alpha \cdot \nu (f_+ - f_-), g_+ + g_- \big)_{L^2(\Sigma)^4}
      - \frac{1}{2} \big( f_+ + f_-, -i c \alpha \cdot \nu (g_+ - g_-)\big)_{L^2(\Sigma)^4}.
    \end{split}
  \end{equation*}
  Using the coupling condition~\eqref{equation_jump_condition_intro} for $f$ and $g$, we conclude that the last term is
  \begin{equation*}
    \begin{split}
      \frac{1}{2} \big( -i c \alpha \cdot \nu (f_+ &- f_-), g_+ + g_- \big)_{L^2(\Sigma)^4} 
          - \frac{1}{2} \big(f_+ + f_-, -i c \alpha \cdot \nu (g_+ - g_-)\big)_{L^2(\Sigma)^4} \\
      &= \left( \frac{1}{4} (\eta I_4 + \tau \beta) (f_+ + f_-), g_+ + g_- \right)_{L^2(\Sigma)^4} \\
      &\qquad \qquad-  \left( f_+ + f_-, \frac{1}{4} (\eta I_4 + \tau \beta) (g_+ + g_-) \right)_{L^2(\Sigma)^4} = 0.
    \end{split}
  \end{equation*}
  Since this is true for any
  $f, g \in \dom A_{\eta, \tau}$, the operator $A_{\eta, \tau}$ is indeed symmetric.
\end{proof}

Next, we prove a Birman-Schwinger principle for the operator $A_{\eta, \tau}$. This relates the linear eigenvalue problem for the differential operator $A_{\eta, \tau}$ to the nonlinear eigenvalue problem for a family of bounded integral operators involving the maps $\mathcal{C}_\lambda$ introduced in~\eqref{def_C_lambda}, which yields also a reduction of the space dimension for the eigenvalue problem. We would like to note that this lemma can only be shown in this simple form for noncritical interaction strengths, i.e. for $\eta^2 - \tau^2 \neq 4 c^2$.
The result stated below follows from the general consideration in \cite[Theorem~2.4]{BEHL17} 
or~\cite[Proposition~3.1]{AMV15}, but to keep the paper self-contained, we add the short simple proof here.

\begin{lem} \label{lemma_Birman_Schwinger}
  Let $\eta, \tau \in \mathbb{R}$ such that $\eta^2 - \tau^2 \neq 4 c^2$ and let the operator $A_{\eta, \tau}$
  be defined by~\eqref{def_A_eta}. 
  \begin{itemize}
    \item[$\textup{(i)}$] If for $\lambda \in \rho(A_0)$ and $\varphi \in L^2(\Sigma)^4$ one has $\big(I_4 + (\eta I_4 + \tau \beta) \mathcal{C}_\lambda \big) \varphi \in H^{1/2}(\Sigma)^4$, then it follows $\varphi \in H^{1/2}(\Sigma)^4$.
    \item[$\textup{(ii)}$] $\lambda \in \rho(A_0) \cap \sigma_{\textup{p}}(A_{\eta, \tau})$ 
    if and only if $-1 \in \sigma_\textup{p}\big((\eta I_4 + \tau \beta) \mathcal{C}_\lambda \big)$.
    \item[$\textup{(iii)}$] For $\lambda \in \mathbb{C} \setminus \mathbb{R}$ the inverse
    \begin{equation*}
      \big( I_4 + (\eta I_4 + \tau \beta) \mathcal{C}_\lambda \big)^{-1}: L^2(\Sigma)^4 \rightarrow L^2(\Sigma)^4
    \end{equation*}
    exists and is bounded and everywhere defined.
  \end{itemize}
\end{lem}
\begin{proof}
(i) If $\big(I_4 + (\eta I_4 + \tau \beta) \mathcal{C}_\lambda \big) \varphi$ belongs to $H^{1/2}(\Sigma)^4$, then by~\eqref{equation_Phi_smooth} also
  \begin{equation*} 
    \begin{split}
      \psi &:= \big( I_4 - \mathcal{C}_\lambda (\eta I_4 - \tau \beta)  \big) \big(I_4 + (\eta I_4 + \tau \beta) \mathcal{C}_\lambda \big) \varphi \\
      &= \left( 1 - \frac{\eta^2 - \tau^2}{4 c^2} \right) \varphi + \tau (\mathcal{C}_\lambda \beta + \beta \mathcal{C}_\lambda) \varphi + (\eta^2 - \tau^2) \left( \frac{1}{4 c^2} I_4 - \mathcal{C}_\lambda^2 \right) \varphi
    \end{split}
  \end{equation*}
  belongs to $H^{1/2}(\Sigma)^4$. Making use of Proposition~\ref{proposition_mapping_properties} this implies that also
  \begin{equation*}
    \varphi = \frac{4 c^2}{4 c^2 - \eta^2 + \tau^2} \left( \psi - \tau (\mathcal{C}_\lambda \beta + \beta \mathcal{C}_\lambda) \varphi - (\eta^2 - \tau^2) \left( \frac{1}{4 c^2} I_4 - \mathcal{C}_\lambda^2 \right) \varphi \right)
  \end{equation*}
  belongs to $H^{1/2}(\Sigma)^4$, which is the claim of item (i).

  (ii) Assume first that $\lambda \in \rho(A_0)$ is an eigenvalue of $A_{\eta, \tau}$ with eigenfunction $f_\lambda \neq 0$.
  Then, according to~\eqref{equation_kernel}
  there exists a density $0 \neq \varphi \in H^{1/2}(\Sigma)^4$ such that 
  $f_\lambda = \Phi_\lambda \varphi$. Since $f_\lambda \in \dom A_{\eta, \tau}$ this function fulfils
  \eqref{equation_jump_condition_intro}.
  Using~\eqref{jump1} and \eqref{jump2} this yields
  \begin{equation} \label{equation_Birman_Schwinger}
    \begin{split}
      0 
      &= i c \alpha \cdot \nu \big( (\Phi_\lambda \varphi)_+ - (\Phi_\lambda \varphi)_-\big) + \frac{1}{2} (\eta I_4 + \tau \beta) \big( (\Phi_\lambda \varphi)_+ + (\Phi_\lambda \varphi)_-\big) \\
      &= (I_4 + (\eta I_4 + \tau \beta) \mathcal{C}_\lambda) \varphi,
    \end{split}
  \end{equation}
  i.e. $-1$ is an eigenvalue of $(\eta I_4 + \tau \beta) \mathcal{C}_\lambda$.

  Conversely, assume that $-1$ is an eigenvalue of $(\eta I_4 + \tau \beta) \mathcal{C}_\lambda$ with  eigenfunction $\varphi \neq 0$. Then it follows first from item~(i) that $\varphi \in H^{1/2}(\Sigma)^4$ and hence $f_\lambda := \Phi_\lambda \varphi \neq 0$ belongs by~\eqref{equation_Phi_smooth} to
  $H^1(\Omega_+)^4 \oplus H^1(\Omega_-)^4$. Using again~\eqref{jump1} and \eqref{jump2} and $\big( I_4 + (\eta I_4 + \tau \beta) \mathcal{C}_\lambda \big) \varphi = 0$ we obtain in the same way as in~\eqref{equation_Birman_Schwinger} that 
  $f_\lambda$ fulfills the coupling condition~\eqref{equation_jump_condition_intro}. This shows $f_\lambda \in \dom A_{\eta, \tau}$.
  Finally, equation~\eqref{equation_kernel} yields 
  \begin{equation*}
    (A_{\eta, \tau} - \lambda) f_\lambda = (A_{\eta, \tau} - \lambda) \Phi_\lambda \varphi = 0
  \end{equation*}
  and hence $\lambda \in \sigma_\textup{p}(A_{\eta, \tau})$.
  
  (iii) To show the claim it suffices to prove that $I_4 + (\eta I_4 + \tau \beta) \mathcal{C}_\lambda$ is bijective. By (ii) it is clear that this operator is injective, as $\lambda \in \mathbb{C} \setminus \mathbb{R}$ and $A_{\eta, \tau}$ is symmetric by Lemma~\ref{lemma_symmetric}. Moreover,
  \begin{equation} \label{equation_C_lambda_range}
    \begin{split}
      \ran \big[I_4 + (\eta I_4 + \tau \beta) \mathcal{C}_\lambda\big]
      &\supset \ran \big[(I_4 + (\eta I_4 + \tau \beta) \mathcal{C}_\lambda)
      (I_4 - (\eta I_4 + \tau \beta) \mathcal{C}_\lambda) \big] \\
      &= \ran \big[I_4 - \big( (\eta I_4 + \tau \beta) \mathcal{C}_\lambda \big)^2 \big]
    \end{split}
  \end{equation}
  holds. Note that $I_4 - \big( (\eta I_4 + \tau \beta) \mathcal{C}_\lambda \big)^2$ is injective, as otherwise $\lambda$ would be a non-real eigenvalue of one of the symmetric operators $A_{\eta, \tau}$ or $A_{-\eta, -\tau}$ by (ii). Moreover,
  \begin{equation*}
    \begin{split}
      I_4 - \big( (\eta I_4 + \tau \beta) \mathcal{C}_\lambda \big)^2
        &= I_4 - \tau (\mathcal{C}_\lambda \beta + \beta \mathcal{C}_\lambda) (\eta I_4 + \tau \beta) \mathcal{C}_\lambda - (\eta^2 - \tau^2) \mathcal{C}_\lambda^2 \\
      &= \left(1 - \frac{\eta^2 - \tau^2}{4 c^2} \right) I_4 + \mathcal{K}_\lambda,
    \end{split}
  \end{equation*}
  where $\mathcal{K}_\lambda$ is a compact operator in $L^2(\Sigma)^4$ by~Proposition~\ref{proposition_mapping_properties}. Therefore, Fredholm's alternative implies that $I_4 - \big( (\eta I_4 + \tau \beta) \mathcal{C}_\lambda \big)^2$ is also surjective. From~\eqref{equation_C_lambda_range} we deduce that the injective operator $I_4 + (\eta I_4 + \tau \beta) \mathcal{C}_\lambda$ is also surjective, which yields the claim of assertion~(iii).
\end{proof}

Now we are prepared to show the self-adjointness of $A_{\eta, \tau}$ in the case of noncritical interaction strengths. Moreover, we prove an explicit Krein type resolvent formula for $A_{\eta, \tau}$ which relates the resolvent of $A_{\eta, \tau}$ to the resolvent of $A_0$ and a perturbation term, which consists of the integral operators $\Phi_\lambda$ and $\mathcal{C}_\lambda$ introduced in~\eqref{def_Phi_lambda} and~\eqref{def_C_lambda}, and contains the spectral information of $A_{\eta, \tau}$.

\begin{thm} \label{theorem_self_adjoint}
  Let $\eta, \tau \in \mathbb{R}$ such that $\eta^2 - \tau^2 \neq 4 c^2$. Then the operator $A_{\eta, \tau}$ defined by~\eqref{def_A_eta} is self-adjoint in $L^2(\mathbb{R}^3)^4$ and 
  \begin{equation*}
    \big(A_{\eta, \tau} - \lambda \big)^{-1} = (A_0 - \lambda)^{-1} - \Phi_\lambda \big( I_4 + (\eta I_4 + \tau \beta) \mathcal{C}_\lambda \big)^{-1} (\eta I_4 + \tau \beta) \Phi_{\bar{\lambda}}^*
  \end{equation*}
  holds for all $\lambda \in \mathbb{C} \setminus \mathbb{R}$.
\end{thm}
\begin{proof}
  We have already shown in Lemma~\ref{lemma_symmetric} that $A_{\eta, \tau}$ is symmetric. Hence, it suffices to prove that $\ran(A_{\eta, \tau} - \lambda) = L^2(\mathbb{R}^3)^4$ for $\lambda \in \mathbb{C} \setminus \mathbb{R}$. Let $\lambda \in \mathbb{C} \setminus \mathbb{R}$ and $f \in L^2(\mathbb{R}^3)^4$ be fixed. We set
  \begin{equation*} 
    g := (A_0 - \lambda)^{-1} f - \Phi_\lambda \big( I_4 + (\eta I_4 + \tau \beta) \mathcal{C}_\lambda \big)^{-1} (\eta I_4 + \tau \beta) \Phi_{\bar{\lambda}}^* f.
  \end{equation*}
  Note that $g$ is well defined by Lemma~\ref{lemma_Birman_Schwinger}~(iii).
  We prove that $g \in \dom A_{\eta, \tau}$ and $(A_{\eta, \tau} - \lambda) g = f$. This shows then $\ran(A_{\eta, \tau} - \lambda) = L^2(\mathbb{R}^3)^4$ and the claimed resolvent formula.
  
  First, we note that $(\eta I_4 + \tau \beta) \Phi_{\bar{\lambda}}^* f \in H^{1/2}(\Sigma)^4$ by~\eqref{equation_Phi_lambda_star} and hence it follows from Lemma~\ref{lemma_Birman_Schwinger}~(i) that
  \begin{equation*}
    \big( I_4 + (\eta I_4 + \tau \beta) \mathcal{C}_\lambda \big)^{-1} (\eta I_4 + \tau \beta) \Phi_{\bar{\lambda}}^* f \in H^{1/2}(\Sigma)^4.
  \end{equation*}
  Thus, we conclude from~\eqref{def_free_Dirac} and~\eqref{equation_Phi_smooth} that $g \in H^1(\Omega_+)^4 \oplus H^1(\Omega_-)^4$.
  
  Next, since $(A_0 - \lambda)^{-1} f \in H^1(\mathbb{R}^3)^4$ the jump of its trace at $\Sigma$ vanishes and we find, using \eqref{jump1}, \eqref{jump2}, and \eqref{equation_Phi_lambda_star}, that
  \begin{equation*}
    \begin{split}
      i c \alpha \cdot \nu& (g_+|_\Sigma - g_-|_\Sigma) + \frac{1}{2} (\eta I_4 + \tau \beta) (g_+|_\Sigma + g_-|_\Sigma)       = (\eta I_4 + \tau \beta) \big( (A_0 - \lambda)^{-1} f\big) \big|_\Sigma \\
      &\quad- \big( I_4 + (\eta I_4 + \tau \beta) \mathcal{C}_\lambda \big) \big( I_4 + (\eta I_4 + \tau \beta) \mathcal{C}_\lambda \big)^{-1} (\eta I_4 + \tau \beta) \Phi_{\bar{\lambda}}^* f = 0,
    \end{split}
  \end{equation*}
  which shows $f \in \dom A_{\eta, \tau}$. Employing finally~\eqref{equation_kernel} we get 
  $(A_{\eta, \tau} - \lambda) g = f$. Hence, the theorem is shown.
\end{proof}

For the self-adjointness of $A_{\eta, \tau}$ in the critical case of interaction strengths, i.e. for $\eta^2 - \tau^2 = 4 c^2$, no result is known so far for combinations of electrostatic and Lorentz scalar interactions. But we would like to review a result from \cite{BH17} (see also \cite{OV17}), where the self-adjointness of $A_{\eta, \tau}$ was shown for purely electrostatic interactions in the critical case, i.e. when $\eta = \pm 2c$ and $\tau = 0$. Already in this simplest example one sees that the properties of $A_{\eta, \tau}$ are completely different in the critical case than in the noncritical case.
The key observation in \cite{BH17} and \cite{OV17} to study the self-adjointness of $A_{\pm 2c, 0}$ is the fact that functions $f \in L^2(\Omega_\pm)^4$ with $\alpha \cdot \nabla f \in L^2(\Omega_\pm)^4$ in the distributional sense have traces in $H^{-1/2}(\Sigma)^4 := (H^{1/2}(\Sigma)^4)'$, which is a larger space than $L^2(\Sigma)^4$. The idea below in~\eqref{jump_condition_critical} is to consider the jump condition~\eqref{equation_jump_condition_intro} not in $L^2(\Sigma)^4$, but in $H^{-1/2}(\Sigma)^4$.

\begin{thm} \label{theorem_critical_case}
  Let $A_{\pm 2c, 0}$ be defined by~\eqref{def_A_eta}. Then $A_{\pm 2c, 0}$ is essentially self-adjoint in $L^2(\mathbb{R}^3)^4$. The self-adjoint closure $\overline{A_{\pm 2c, 0}}$ is defined on the set
  \begin{equation} \label{jump_condition_critical}
    \begin{split}
      \dom \overline{A_{\pm 2c, 0}} = \big\{ &f = f_+ \oplus f_- \in L^2(\Omega_+)^4 \oplus L^2(\Omega_-)^4: \alpha \cdot \nabla f_\pm \in L^2(\Omega_\pm)^4, \\
      &i \alpha \cdot \nu (f_+|_\Sigma - f_-|_\Sigma) = \mp (f_+|_\Sigma + f_-|_\Sigma) \text{ in } H^{-1/2}(\Sigma)^4 \big\}
    \end{split}
  \end{equation}
  and acts as
  \begin{equation*}
    \overline{A_{\pm 2c, 0}} f = (-i c \alpha \cdot \nabla + m c^2 \beta) f_+ \oplus (-i c \alpha \cdot \nabla + m c^2 \beta) f_-.
  \end{equation*}
  The closure $\overline{A_{\pm 2c, 0}}$ is a proper extension of $A_{\pm 2c, 0}$, i.e. $A_{\pm 2c, 0} \neq \overline{A_{\pm 2c, 0}}$.
\end{thm}

\section{Spectral properties} \label{section_spectral_properties}

In this section we provide the basic spectral properties of the operator $A_{\eta, \tau}$ defined by~\eqref{def_A_eta}. In the case of noncritical interaction strengths, i.e. when $\eta^2 - \tau^2 \neq 4 c^2$, we are able to provide a number of results about the qualitative spectral properties. We close this section with a result from \cite{BH17} on the spectrum of the self-adjoint closure of $A_{\pm 2c, 0}$ in the case of purely electrostatic critical interactions, which shows that the spectral properties for critical interaction strengths can be of a completely different nature.

First, we discuss the basic results in the noncritical case. In particular, using a perturbation argument based on the Krein type resolvent formula from Theorem~\ref{theorem_self_adjoint} we compute the essential spectrum of $A_{\eta, \tau}$. Moreover, since the singular perturbation is only supported on a compact surface and since functions in $\dom A_{\eta, \tau}$ have $H^1$-smoothness, we can show that $A_{\eta, \tau}$ has only finitely many discrete eigenvalues. Eventually, we deduce from the Birman-Schwinger principle that $A_{\eta, \tau}$ has no discrete eigenvalues, if the interaction strengths are sufficiently small.

\begin{thm} \label{theorem_spectral_properties}
  Let $\eta, \tau \in \mathbb{R}$ such that $\eta^2 - \tau^2 \neq 4 c^2$
  and let the self-adjoint operator $A_{\eta, \tau}$ be defined by~\eqref{def_A_eta}.
  Then the following assertions hold.
  \begin{itemize}
    \item[$\textup{(i)}$] $\sigma_{\textup{ess}}(A_{\eta, \tau}) = (-\infty, -m c^2] \cup [m c^2, \infty)$.
    \item[$\textup{(ii)}$] $\sigma_{\textup{disc}}(A_{\eta, \tau})$ is finite.
    \item[$\textup{(iii)}$] There exists a constant $K > 0$ such that $\sigma_{\textup{disc}}(A_{\eta, \tau}) = \emptyset$, if $| \eta + \tau | < K$ and $| \eta - \tau | < K$.
  \end{itemize}
\end{thm}
\begin{proof}
  In order to show item~(i) we note that for $\lambda \in \mathbb{C} \setminus \mathbb{R}$ 
  the operators 
  \begin{equation*}
    \Phi_\lambda, \Phi_{\bar{\lambda}}^*, \quad \text{and} \quad
    \big( I_4 + (\eta I_4 + \tau \beta) \mathcal{C}_\lambda \big)^{-1}
  \end{equation*}
  are bounded in the respective 
  $L^2$-spaces, see Lemma~\ref{lemma_Birman_Schwinger}. Moreover, it follows from~\eqref{equation_Phi_lambda_star} and the trace theorem~\eqref{trace_theorem} that $\Phi_{\bar{\lambda}}^*$  is bounded from $L^2(\mathbb{R}^3)^4$ to $H^{1/2}(\Sigma)^4$ and since $H^{1/2}(\Sigma)^4$ is compactly embedded in $L^2(\Sigma)^4$, see \cite[Theorem~4.2.2]{HW08}, we get that $\Phi_{\bar{\lambda}}^*$ is compact. Hence, using the resolvent formula from Theorem~\ref{theorem_self_adjoint} we conclude that  
  \begin{equation*}
    (A_{\eta, \tau} - \lambda)^{-1} - (A_0 - \lambda)^{-1}
        = -\Phi_\lambda \big( I_4 + (\eta I_4 + \tau \beta) \mathcal{C}_\lambda \big)^{-1} 
            (\eta I_4 + \tau \beta) \Phi_{\bar{\lambda}}^*
  \end{equation*}
  is compact in $L^2(\mathbb{R}^3)^4$. Therefore, it follows from \cite[Theorem~XIII.14]{RS78} that
  $\sigma_{\textup{ess}}(A_{\eta, \tau}) = \sigma_{\textup{ess}}(A_0) = (-\infty, -m c^2] \cup [m c^2, \infty)$.
  
  The proof of statement~(ii) follows ideas from \cite[Proposition~3.6]{HOP17}.
  We note first that the number of discrete eigenvalues of 
  $A_{\eta, \tau}$ in the gap $(-m c^2, m c^2)$ is equal to the number of 
  eigenvalues of $(A_{\eta, \tau})^2$ below the threshold of its essential spectrum $(m c^2)^2$.
  Let us denote the quadratic form associated to $(A_{\eta, \tau})^2$ by $\mathfrak{a}$.
  Then for any $f = f_+ \oplus f_- \in \dom A_{\eta, \tau} = \dom \mathfrak{a}$
  \begin{equation*}
    \begin{split}
      \mathfrak{a}[f] &= \| A_{\eta, \tau} f \|_{L^2(\mathbb{R}^3)^4}^2 \\
      &= \big\| (-i c \alpha \cdot \nabla + m c^2 \beta) f_+ \big\|_{L^2(\Omega_+)^4}^2
        + \big\| (-i c \alpha \cdot \nabla + m c^2 \beta) f_- \big\|_{L^2(\Omega_-)^4}^2 \\
      &= \| c (\alpha \cdot \nabla) f_+ \|_{L^2(\Omega_+)^4}^2 + \| c (\alpha \cdot \nabla) f_- \|_{L^2(\Omega_-)^4}^2
          + (m c^2)^2 \| f \|_{L^2(\mathbb{R}^3)^4}^2 \\
      &\qquad + ( -i c \alpha \cdot \nabla f_+, m c^2 \beta f_+ )_{L^2(\Omega_+)^4} 
          + ( m c^2 \beta f_+, -i c \alpha \cdot \nabla f_+ )_{L^2(\Omega_+)^4} \\
      &\qquad + (- i c \alpha \cdot \nabla f_-, m c^2 \beta f_- )_{L^2(\Omega_-)^4} 
          + ( m c^2 \beta f_-, -i c \alpha \cdot \nabla f_- )_{L^2(\Omega_-)^4}
    \end{split}
  \end{equation*}
  holds. Employing integration by parts and \eqref{equation_anti_commutation} we see that
  \begin{equation*}
    \begin{split}
     ( -i c \alpha \cdot \nabla f_\pm, m c^2 \beta f_\pm )_{L^2(\Omega_\pm)^4} 
          &+ ( m c^2 \beta f_\pm, -i c \alpha \cdot \nabla f_\pm )_{L^2(\Omega_\pm)^4} \\
          &= \mp (i c \alpha \cdot \nu f_\pm|_\Sigma, m c^2 \beta f_\pm|_\Sigma)_{L^2(\Sigma)^4},
    \end{split}
  \end{equation*}
  which yields then
  \begin{equation*} 
    \begin{split}
      \mathfrak{a}[f] 
        &= \| c (\alpha \cdot \nabla) f \|_{L^2(\Omega_+ \cup \Omega_-)^4}^2 + (m c^2)^2 \| f \|_{L^2(\mathbb{R}^3)^4}^2 \\
        &\quad  - (i c \alpha \cdot \nu f_+|_\Sigma, m c^2 \beta f_+|_\Sigma)_{L^2(\Sigma)^4}
          + (i c \alpha \cdot \nu f_-|_\Sigma, m c^2 \beta f_-|_\Sigma)_{L^2(\Sigma)^4}.
    \end{split}
  \end{equation*}
  To proceed choose $R > 0$ such that $\Sigma \subset B(0, R)$
  and define the closed and semibounded sesquilinear forms $\mathfrak{b}_\text{int}$ and $\mathfrak{b}_\text{ext}$ by
  \begin{equation*}
    \begin{split}
      \mathfrak{b}_{\textup{int}}[f] &:=  \| c (\alpha \cdot \nabla) f\|_{L^2(\Omega_+ \cup (\Omega_- \cap B(0, R)))^4}^2
          + (m c^2)^2\| f\|_{L^2(B(0, R))^4}^2 \\
      &\quad - (i c \alpha \cdot \nu f_+|_\Sigma, m c^2 \beta f_+|_\Sigma)_{L^2(\Sigma)^4}
          + (i c \alpha \cdot \nu f_-|_\Sigma, m c^2 \beta f_-|_\Sigma)_{L^2(\Sigma)^4}, \\
      \dom \mathfrak{b}_{\textup{int}}
          &:= \bigg\{ f = f_+ \oplus f_- \in H^1(\Omega_+)^4 \oplus H^1(\Omega_- \cap B(0, R))^4: \\
      &\qquad \quad i c (\alpha \cdot \nu) (f_+|_\Sigma - f_-|_\Sigma) 
          = -\frac{1}{2} (\eta I_4 + \tau \beta) (f_+|_\Sigma + f_-|_\Sigma) \bigg\},
    \end{split}
  \end{equation*}
  and
  \begin{equation*}
    \begin{split}
      \mathfrak{b}_{\textup{ext}}[f] &:= \| c (\alpha \cdot \nabla) f\|_{L^2(\mathbb{R}^3 \setminus B(0, R))^4}^2
          + (m c^2)^2\| f\|_{L^2(\mathbb{R}^3 \setminus B(0, R))^4}^2, \\
      \dom \mathfrak{b}_{\textup{ext}} &:= H^1(\mathbb{R}^3 \setminus B(0, R))^4.
    \end{split}
  \end{equation*}
  Then $\mathfrak{a}$
  is minorated in the sense of closed quadratic forms by $\mathfrak{b} := \mathfrak{b}_{\textup{int}} \oplus \mathfrak{b}_{\textup{ext}}$, that means $\dom \mathfrak{a} \subset \dom \mathfrak{b}$ and $\mathfrak{b}[f] \leq \mathfrak{a}[f]$ for all $f \in \dom \mathfrak{a}$. By the min-max principle this implies that, if the operator associated 
  to $\mathfrak{b}$ has finitely many eigenvalues below $(m c^2)^2$, then $(A_{\eta, \tau})^2$
  has only finitely many eigenvalues below $(m c^2)^2$.
  
  Clearly, the operator $B_{\textup{ext}}$ associated to $\mathfrak{b}_{\textup{ext}}$ is bounded from below by $B_{\textup{ext}} \geq (m c^2)^2$. Thus, the number of eigenvalues of $(A_{\eta, \tau})^2$
  below $(m c^2)^2$ is less or equal to the number of eigenvalues of the operator $B_{\textup{int}}$ 
  associated to the semibounded and closed form $\mathfrak{b}_{\textup{int}}$, 
  compare for instance \cite[Section XIII.15]{RS78} for a similar argument.
  Moreover, as $\dom \mathfrak{b}_{\textup{int}} \subset H^1(\Omega_+)^4 \oplus H^1(\Omega_- \cap B(0, R))^4$
  is compactly embedded in $L^2(B(0, R))^4$, cf. \cite[Theorem~3.27]{M00}, it follows that the resolvent of 
  $B_{\textup{int}}$ is compact.
  Therefore, the spectrum of $B_{\textup{int}}$ is purely discrete and consists of eigenvalues that 
  accumulate only at $\infty$, as $B_{\textup{int}}$ is bounded from below. 
  Thus $B_{\textup{int}}$ has only finitely many eigenvalues below $(m c^2)^2$. Hence, also the operator associated to
  $\mathfrak{b}$ has only finitely many eigenvalues below $(m c^2)^2$. This shows finally that
  $(A_{\eta, \tau})^2$ has only finitely many eigenvalues below $(m c^2)^2$
  which finishes the proof of assertion~(ii). 
  
  Finally, item~(iii) is just a simple consequence of the Birman-Schwinger principle in Lemma~\ref{lemma_Birman_Schwinger}~(ii)
  and~\eqref{C_lambda_uniformly_bounded}.
\end{proof}

As it is often the case for Dirac operators we also have several symmetry relations for the spectrum of $A_{\eta, \tau}$. These symmetries are consequences of commutation relations of $A_{\eta, \tau}$ with the charge conjugation, the time reversal and a suitable unitary operator. We would like to note that item~(i) in the proposition below can also be shown with the aid of the Birman-Schwinger principle from Lemma~\ref{lemma_Birman_Schwinger}, cf. the proof of \cite[Theorem~3.3]{AMV15} for the purely electrostatic case. The presentation below follows \cite[Theorem~2.3]{HOP17}, where the special case of purely scalar interactions is treated.

\begin{prop} \label{proposition_discrete_spectrum}
  Let $\eta, \tau \in \mathbb{R}$ such that
  $\eta^2 - \tau^2 \neq 4 c^2$. Then the following is true.
  \begin{itemize}
    \item[$\textup{(i)}$] Assume $\eta^2 \neq \tau^2$. Then $\lambda \in \sigma_{\textup{p}}
    \big(A_{-4c^2 \eta/(\eta^2 - \tau^2), -4 c^2 \tau/(\eta^2 - \tau^2)}\big)$ if and only if $\lambda \in \sigma_{\textup{p}}(A_{\eta, \tau})$.
    \item[$\textup{(ii)}$] $\lambda \in \sigma_{\textup{p}}(A_{\eta, \tau})$ has always even multiplicity.
    \item[$\textup{(iii)}$] $\lambda \in \sigma_{\textup{p}}(A_{\eta, \tau})$ if and only if
    $-\lambda \in \sigma_{\textup{p}}(A_{-\eta, \tau})$.
  \end{itemize}
\end{prop}
\begin{proof}
  (i) Assume that $f = f_+ \oplus f_-$ is an eigenfunction of $A_{\eta, \tau}$ for the eigenvalue~$\lambda$. Then the function $g := f_+ \oplus (-f_-) \in H^1(\Omega_+)^4 \oplus H^1(\Omega_-)^4$ fulfils
  \begin{equation*}
    \begin{split}
      -i c \alpha \cdot \nu &(g_+|_\Sigma + g_-|_\Sigma) = -i c \alpha \cdot \nu (f_+|_\Sigma - f_-|_\Sigma) \\
        &= \frac{1}{2} (\eta I_4 + \tau \beta) (f_+ |_\Sigma + f_-|_\Sigma)
        = \frac{1}{2} (\eta I_4 + \tau \beta) (g_+ |_\Sigma - g_-|_\Sigma),
    \end{split}
  \end{equation*}
  as $f \in \dom A_{\eta, \tau}$. A multiplication of the last equation with the constant matrix $(\eta I_4 + \tau \beta)^{-1} = \frac{1}{\eta^2 - \tau^2} (\eta I_4 - \tau \beta)$ yields
  \begin{equation*}
      - \frac{i c}{\eta^2 - \tau^2} (\eta I_4 - \tau \beta) \alpha \cdot \nu (g_+|_\Sigma + g_-|_\Sigma) 
        = \frac{1}{2} (g_+ |_\Sigma - g_-|_\Sigma).
  \end{equation*}
  By using the anti-commutation relation~\eqref{equation_anti_commutation} and multiplying this equation then with $- 2 i c \alpha \cdot \nu$ one easily sees that 
  \begin{equation*}
      - \frac{1}{2} \frac{4 c^2}{\eta^2 - \tau^2} (\eta I_4 + \tau \beta) (g_+|_\Sigma + g_-|_\Sigma) 
        = -i c \alpha \cdot \nu (g_+ |_\Sigma - g_-|_\Sigma),
  \end{equation*}
  which shows that $g \in \dom A_{-4c^2 \eta/(\eta^2 - \tau^2), -4 c^2 \tau/(\eta^2 - \tau^2)}$. Finally, since $f$ is an eigenfunction of $A_{\eta, \tau}$ corresponding to $\lambda$ one deduces immediately that also 
  \begin{equation*}
    A_{-4c^2 \eta/(\eta^2 - \tau^2), -4 c^2 \tau/(\eta^2 - \tau^2)} g = \lambda g,
  \end{equation*}
  which shows item (i).

  For the proof of statement~(ii) we define the (nonlinear) time reversal operator
  \begin{equation*}
    T f := -i \gamma_5 \alpha_2 \overline{f}, \qquad f \in L^2(\mathbb{R}^3)^4, \quad 
    \gamma_5 := \begin{pmatrix} 0 & I_2 \\ I_2 & 0 \end{pmatrix}.
  \end{equation*}
  Note that $\beta \gamma_5 = - \gamma_5 \beta$ and $(\alpha \cdot x) \gamma_5 = \gamma_5 (\alpha \cdot x)$
  for any $x \in \mathbb{R}^3$. First we show that $f \in \dom A_{\eta, \tau}$ implies
  $T f \in \dom A_{\eta, \tau}$. Indeed, if one takes for $f \in \dom A_{\eta, \tau}$ the complex conjugate of the coupling condition~\eqref{equation_jump_condition_intro} and multiplies this equation with the matrix $-i \gamma_5 \alpha_2$ we deduce
  \begin{equation*}
    -i \gamma_5 \alpha_2 i c \overline{\alpha} \cdot \nu(\overline{f_+}|_\Sigma - \overline{f_-}|_\Sigma)
      = -\frac{i}{2} \gamma_5 \alpha_2 (\eta I_4 + \tau \beta) (\overline{f_+}|_\Sigma + \overline{f_-}|_\Sigma).
  \end{equation*}
  Using $\overline{\alpha_2} = -\alpha_2$ (where the complex conjugate is understood component wise) and~\eqref{equation_anti_commutation} we deduce from the last equality that also $T f$ satisfies \eqref{equation_jump_condition_intro} and hence $T f \in \dom A_{\eta, \tau}$.
  
  Employing again $\overline{\alpha_2} = -\alpha_2$ one finds $T^2 f= -f$. Furthermore, using
  \eqref{equation_anti_commutation} we get 
  \begin{equation} \label{equation_time_reversal}
    \begin{split}
      (-i c &\alpha \cdot \nabla + m c^2 \beta) T f 
          = (-i c \alpha \cdot \nabla + m c^2 \beta) ( -i \gamma_5 \alpha_2 \overline{f}) \\
      &=  -i \gamma_5 \alpha_2 (i c \overline{\alpha} \cdot \nabla + m c^2 \beta) \overline{f}
          = T \big((-i c \alpha \cdot \nabla + m c^2 \beta) f \big).
    \end{split}
  \end{equation}
  This shows $A_{\eta, \tau} T f = T A_{\eta, \tau} f$ for $f \in \dom A_{\eta, \tau}$. Another calculation using again $\overline{\alpha_2} = -\alpha_2$ gives 
  $\langle -i \gamma_5 \alpha_2 \overline{f}, f \rangle_{\mathbb{C}^4} 
  = \overline{\langle f, i \gamma_5 \alpha_2 \overline{f} \rangle_{\mathbb{C}^4}}$ which implies
  \begin{equation*}
    (T f, f)_{L^2(\mathbb{R}^3)^4} = \int_{\mathbb{R}^3} T f(x) \cdot \overline{f(x)} \text{d} x = 0.
  \end{equation*}
  Hence, if $f$ is an eigenfunction of $A_{\eta, \tau}$, then also
  $T f \in \dom A_{\eta, \tau}$ is a linearly independent and non-trivial eigenfunction of $A_{\eta, \tau}$ for the same eigenvalue.  
  Therefore, also assertion~(ii) is proven.

  Eventually, to prove statement~(iii) we introduce the (nonlinear) charge conjugation operator
  \begin{equation*}
    C f := i \beta \alpha_2 \overline{f}, \qquad f \in L^2(\mathbb{R}^3)^4.
  \end{equation*}
  A simple calculation similar as above shows $C^2 f = f$. Moreover,
  it is not difficult to see that $f \in \dom A_{\eta, \tau}$ 
  if and only if $C f \in \dom A_{-\eta, \tau}$. Finally, a similar calculation as 
  in~\eqref{equation_time_reversal} shows
  \begin{equation*}
    (-i c \alpha \cdot \nabla + m c^2 \beta) C f = -C (-i c \alpha \cdot \nabla + m c^2 \beta) f.
  \end{equation*}
  Hence, we deduce $f \in \dom A_{\eta, \tau}$ fulfils $A_{\eta, \tau} f = \lambda f$ if and only if $Cf \in \dom A_{-\eta, \tau}$ and $A_{-\eta, \tau} C f = -\lambda C f$.
  This yields then the claim of item~(iii).
\end{proof}

By combining Theorem~\ref{theorem_spectral_properties}~(iii) with Proposition~\ref{proposition_discrete_spectrum} we find that $A_{\eta, \tau}$ does not have discrete eigenvalues also for large interaction strengths. This is in contrast to what is known for Schr\"odinger operators with singular $\delta$-potentials. For the nonrelativistic Hamiltonians with attractive $\delta$-interactions in~$\mathbb{R}^3$ there are no eigenvalues for small interaction strengths~\cite{EF09}, but always eigenvalues for large values of the interaction strength~\cite{E03}. The difference is obviously due to the presence of the `lower continuum' for the Dirac operator.

\begin{cor} \label{corollary_no_discrete_spectrum}
  Let $\eta, \tau \in \mathbb{R}$ such that $\eta^2 - \tau^2 \notin \{ 0, 4 c^2\}$ and let $K$ be the same constant as in Theorem~\ref{theorem_spectral_properties}~(iii).
  Then $\sigma_{\textup{disc}}(A_{\eta, \tau}) = \emptyset$, if $| \eta + \tau | > \frac{4 c^2}{K}$ and $| \eta - \tau | > \frac{4 c^2}{K}$.
\end{cor}

Theorem~\ref{theorem_spectral_properties}, Lemma~\ref{lemma_Birman_Schwinger}, and Proposition~\ref{proposition_discrete_spectrum} give a detailed picture of the spectral properties of $A_{\eta, \tau}$. For purely electrostatic and purely Lorentz scalar interactions, which are the most interesting ones of the potentials considered here for applications in relativistic quantum mechanics, many of these findings simplify significantly; hence, we summarize the spectral properties for these two important cases in the following corollaries. We start with the purely electrostatic case:

\begin{cor} \label{corollary_electrostatic_interaction}
  Let $\eta \in \mathbb{R} \setminus \{ \pm 2c \}$. Then the following assertions hold.
  \begin{itemize}
    \item[$\textup{(i)}$] For $\lambda \in \mathbb{C} \setminus \mathbb{R}$ the resolvent of $A_{\eta, 0}$
    is given by
    \begin{equation*}
      \begin{split}
        (A_{\eta, 0}-\lambda)^{-1} = (A_0 - \lambda)^{-1} 
            - \Phi_\lambda \big( I_4 + \eta \mathcal{C}_\lambda \big)^{-1} \eta \Phi_{\bar{\lambda}}^*.
      \end{split}
    \end{equation*}
    \item[$\textup{(ii)}$] $\sigma_{\textup{ess}}(A_{\eta, 0}) = (-\infty, -m c^2] \cup [m c^2, \infty)$.
    \item[$\textup{(iii)}$] $\lambda \in \sigma_{\textup{p}}(A_{\eta, 0}) \cap (-mc^2, mc^2)$ if and only if
    $-1 \in \sigma_{\textup{p}}(\eta \mathcal{C}_\lambda)$.
    \item[$\textup{(iv)}$] If $\eta \neq 0$, then
    $\lambda \in \sigma_{\textup{p}}(A_{\eta, 0})$ if and only if
    $\lambda \in \sigma_{\textup{p}}\big(A_{-4c^2/\eta, 0}\big)$.
    \item[$\textup{(v)}$] $\lambda \in \sigma_{\textup{p}}(A_{\eta, 0})$ if and only if
    $-\lambda \in \sigma_{\textup{p}}\big(A_{-\eta, 0}\big)$.
    \item[$\textup{(vi)}$] $\sigma_{\textup{disc}}(A_{\eta, 0})$ is finite.
    \item[$\textup{(vii)}$] Eigenvalues of $A_{\eta, 0}$ have always even multiplicity.
    \item[$\textup{(viii)}$] There exists a constant $K > 0$ such that $\sigma_{\textup{disc}}(A_{\eta, 0}) = \emptyset$,
    if $| \eta| < K$ or $|\eta| > \frac{4 c^2}{K}$.
  \end{itemize} 
\end{cor}

Next, let us discuss Dirac operators with purely Lorentz scalar $\delta$-shell interactions,
that means we assume $\eta = 0$. 
Note that in this case there is no critical interaction strength, as 
$-\tau^2 \neq 4 c^2$ always in this case. 
On the other hand we have confinement for $\tau = \pm 2 c$, compare Lemma~\ref{lemma_transmission_condition}.
Note that most of the results below are also formulated and proved in \cite[Theorem~2.3]{HOP17}.

\begin{cor} \label{corollary_scalar_interaction}
  Let $\tau \in \mathbb{R}$. Then the following assertions hold.
  \begin{itemize}
    \item[$\textup{(i)}$] For $\lambda \in \mathbb{C} \setminus \mathbb{R}$ the resolvent of $A_{0, \tau}$
    is given by
    \begin{equation*}
      \begin{split}
        (A_{0, \tau}-\lambda)^{-1} = (A_0 - \lambda)^{-1} 
            - \Phi_\lambda \big( I_4 + \tau \beta \mathcal{C}_\lambda \big)^{-1}
            \tau \beta \Phi_{\bar{\lambda}}^*.
      \end{split}
    \end{equation*}
    \item[$\textup{(ii)}$] $\sigma_{\textup{ess}}(A_{0, \tau}) = (-\infty, -m c^2] \cup [m c^2, \infty)$.
    \item[$\textup{(iii)}$] $\lambda \in \sigma_{\textup{p}}(A_{0, \tau}) \cap(-mc^2, mc^2)$ if and only if
    $-1 \in \sigma_{\textup{p}}(\tau \beta \mathcal{C}_\lambda)$.
    \item[$\textup{(iv)}$] If $\tau \neq 0$, then $\lambda \in \sigma_{\textup{p}}(A_{0, \tau})$ if and only if
    $\lambda \in \sigma_{\textup{p}}\big(A_{0, 4c^2/\tau}\big)$.
    \item[$\textup{(v)}$] $\lambda \in \sigma_{\textup{p}}(A_{0, \tau})$ if and only if $-\lambda \in \sigma_{\textup{p}}(A_{0, \tau})$.
    \item[$\textup{(vi)}$] $\sigma_{\textup{disc}}(A_{0, \tau})$ is finite.
    \item[$\textup{(vii)}$] Eigenvalues of $A_{0, \tau}$ have always even multiplicity.
    \item[$\textup{(viii)}$] There exists a constant $K > 0$ such that $\sigma_{\textup{disc}}(A_{0, \tau}) = \emptyset$,
    if $| \tau| < K$ or $|\tau| > \frac{4 c^2}{K}$.
  \end{itemize} 
\end{cor}

In addition to what we know about the spectrum of $A_{0, \tau}$ in the purely scalar case from Corollary~\ref{corollary_scalar_interaction} an explicit formula for the quadratic form associated to $A_{0, \tau}^2$ is shown in \cite[Proposition~3.1]{HOP17}. This formula implies also that there are no discrete eigenvalues of $A_{0, \tau}$ for $\tau \geq 0$. For a further discussion of consequences of this interesting result we refer the reader to \cite{HOP17}.

\begin{cor} \label{corollary_scalar_interaction_form}
  Let $\tau \in \mathbb{R} \setminus \{ \pm 2 c\}$ and assume that $\Sigma$ is $C^4$-smooth. Then the following assertions hold.
  \begin{itemize}
    \item[$\textup{(i)}$] If $\tau \neq 0$, then for any $f \in \dom A_{0, \tau}$ 
    \begin{equation*}
      \begin{split}
          \| A_{0,\tau} f \|^2_{L^2(\mathbb{R}^3)^4} &= c^2 \int_{\mathbb{R}^3 \setminus \Sigma} \big|\nabla f \big|^2 \textup{d} x + (m c^2)^2 \int_{\mathbb{R}^3}  |f|^2 \textup{d} x +  c^2 \int_\Sigma M \big|f_+|_\Sigma\big|^2 \textup{d} \sigma\\
         &\qquad  - c^2 \int_\Sigma M \big|f_-|_\Sigma\big|^2 \textup{d} \sigma + \frac{2 m c^4}{\tau} \int_\Sigma \big|f_+|_\Sigma - f_-|_\Sigma\big|^2 \textup{d} \sigma
      \end{split}
    \end{equation*}
    holds, where $M$ is the mean curvature at $\Sigma$. 
    \item[$\textup{(ii)}$] If $\tau \geq 0$, then $\sigma_{\textup{disc}}(A_{0, \tau}) = \emptyset$.
  \end{itemize} 
\end{cor}

Eventually, we state that the difference of the third powers of the resolvents of $A_{\eta, \tau}$ and $A_0$ is a trace class operator. This result is of interest for mathematical scattering theory,
as it ensures the existence and completeness of the wave operators for the scattering system 
$\{ A_{\eta, \tau}, A_0 \}$ and implies that the absolutely continuous parts of 
$A_{\eta, \tau}$ and $A_0$ are unitarily equivalent, cf. \cite[Chapter~0, Theorem~8.2]{Y10} 
and the standard definition of existence and completeness of wave operators. The proof of this result in the purely electrostatic case, i.e. when $\tau = 0$, can be found in \cite[Theorem~4.6]{BEHL17}, in the general case one can follow it almost word by word. Hence, we omit the proof here. Note that we have to assume some additional smoothness of $\Sigma$ here to ensure that the result is correct.

\begin{prop} \label{theorem_Schatten_von_Neumann}
  Assume that $\Sigma$ is $C^\infty$-smooth and let $\eta, \tau \in \mathbb{R}$ such that
  $\eta^2 - \tau^2 \neq 4 c^2$. Then for any $\lambda \in \mathbb{C} \setminus \mathbb{R}$ the operator
  \begin{equation*}
    (A_{\eta, \tau} - \lambda)^{-3} - (A_0 - \lambda)^{-3}
  \end{equation*}
  belongs to the trace class. In particular, the wave operators for the system $\{ A_{\eta, \tau}, A_0 \}$ exist and are complete, and the absolutely continuous parts of $A_{\eta, \tau}$ and $A_0$ are unitarily equivalent
\end{prop}

Finally, we formulate a result shown in \cite[Section~5]{BH17} about the spectral properties of $A_{\eta, \tau}$ in the case of critical interaction strengths. Again, the result is only known for purely electrostatic interactions, i.e. for $\eta = \pm 2 c$ and $\tau=0$. Nevertheless, the theorem below shows that the spectral properties of $A_{\eta, \tau}$ can be of a completely different type for the critical interaction strengths.
To formulate the result, we say that a surface $\Sigma$ contains a {\it flat part}, if there exists an open $\Sigma_0 \subset \Sigma$ such that $\Sigma_0$ is contained in a plane in $\mathbb{R}^3$. The complete proof of the following theorem as well as further results on the spectrum, a variant of Krein's resolvent formula, and the Birman-Schwinger principle for the self-adjoint closure of $A_{\pm 2c, 0}$ can be found in~\cite{BH17}.

\begin{thm} \label{theorem_spectrum_critical_case}
  Let $A_{\pm 2c, 0}$ be defined by~\eqref{def_A_eta}. Then $(-\infty, -m c^2] \cup [m c^2, \infty)$ belongs to $\sigma_\textup{ess}(\overline{A_{\pm 2c, 0}})$. If $\Sigma$ contains a flat part, then also $0 \in \sigma_\textup{ess}(\overline{A_{\pm 2c, 0}})$.
\end{thm}

\section{Nonrelativistic limit}
\label{section_nonrelativistic_limit_singular_interaction}

In this section we study the nonrelativistic limit of Dirac operators with purely electrostatic 
or purely Lorentz scalar $\delta$-shell interactions, that means we study this limit of $A_{\eta, \tau}$
in the cases that either $\tau = 0$ or $\eta = 0$ which are of particular physical interest.
In the nonrelativistic limit one subtracts/adds the energy of the mass of the particle $m c^2$ from the total 
energy and computes the limit of the resolvent, as $c \rightarrow \infty$.
The expected result is the resolvent of a nonrelativistic Schr\"odinger operator
which describes the same physical problem with the same parameters times a projection onto the upper/lower 
components of the Dirac wave function. 
In our case we will see that the Dirac operator with an electrostatic or a scalar $\delta$-shell interaction
converges in the nonrelativistic limit to a Schr\"odinger operator with a $\delta$-potential of the same strength.
This gives a further justification for the usage of the operator $A_{\eta, 0}$ and $A_{0, \tau}$
as a Dirac operator with a singular $\delta$-interaction supported on $\Sigma$.
The presentation in this section follows closely \cite[Section~5]{BEHL17}.

First, we introduce some notations which are necessary to formulate the main result of this section, afterwards we discuss shortly the idea of the proof. As usual let $\Sigma \subset \mathbb{R}^3$ be the boundary of a compact $C^2$-domain. We define for $\eta \in \mathbb{R}$ the sesquilinear form
\begin{equation*} 
  \mathfrak{a}_\eta[f, g] := \frac{1}{2 m} (\nabla f, \nabla g)_{L^2(\mathbb{R}^3)^3} + (\eta f|_\Sigma, g|_\Sigma)_{L^2(\Sigma)},\quad
  f, g \in \dom \mathfrak{a}_\eta := H^1(\mathbb{R}^3).
\end{equation*}
It is not difficult to show that $\mathfrak{a}_\eta$ is symmetric, semibounded from below and closed,
see for instance \cite[Section~4]{BEKS94} or \cite{BLL13}. The associated self-adjoint operator $-\Delta_\eta$ is
\begin{equation*}
  \begin{split}
    -\Delta_\eta f &= \left( -\frac{1}{2 m} \Delta f_+ \right) \oplus \left(  -\frac{1}{2 m} \Delta f_+\right), \\
    \dom (-\Delta_\eta) &= \big\{ f = f_+ \oplus f_- \in \big(H^2(\Omega_+) \oplus H^2(\Omega_-)\big) \cap H^1(\mathbb{R}^3): \\
    & \qquad \qquad \qquad \qquad \qquad \qquad \qquad 2 m \eta f|_\Sigma = \partial_\nu f_-|_\Sigma - \partial_\nu f_+|_\Sigma \big\},
  \end{split}
\end{equation*}
where $H^2(\Omega_\pm)$ is the Sobolev space containing all functions for which the first and the second distributional derivatives belong to $L^2(\Omega_\pm)$,
and it is the Schr\"odinger operator with a $\delta$-potential of strength $\eta$ supported on $\Sigma$, i.e. formally it holds $-\Delta_\eta = -\frac{1}{2m} \Delta + \eta \delta_\Sigma$; cf. \cite[Section~3.2]{BLL13}. 
Next, we set
\begin{equation*}
  P_+ := \begin{pmatrix} I_2 & 0 \\ 0 & 0 \end{pmatrix} \quad \text{and} \quad
  P_- := \begin{pmatrix} 0 & 0 \\ 0 & I_2 \end{pmatrix}.
\end{equation*}
The following theorem treats the nonrelativistic limit of $A_{\eta, 0}$ and $A_{0, \tau}$. In particular, it shows that these operators are indeed the relativistic counterparts of $-\Delta_\eta$ with electrostatic and Lorentz scalar interactions, respectively. Note that the result holds for any $\eta \in \mathbb{R}$, as $4 c^2 > \eta^2$ for all sufficiently large $c$, and hence, we do not have to take care of the critical interaction strengths.

\begin{thm} \label{theorem_nonrelativistic_limit}
  For any $\eta, \tau \in \mathbb{R}$ and all $\lambda \in \mathbb{C} \setminus \mathbb{R}$
  there exists a constant $K > 0$ such that for all sufficiently large $c$
  \begin{equation*}
    \left\| \big(A_{\eta, 0} - (\lambda + m c^2)\big)^{-1} - (-\Delta_\eta - \lambda)^{-1} P_+ \right\| \leq \frac{K}{c}
  \end{equation*}
  and
  \begin{equation*}
    \left\| \big(A_{0, \tau} - (\lambda \pm m c^2)\big)^{-1} - \big(\pm(-\Delta_\tau) - \lambda\big)^{-1} P_\pm \right\| 
        \leq \frac{K}{c}.
  \end{equation*}
\end{thm}

An interesting aspect in Theorem~\ref{theorem_nonrelativistic_limit} is the fact that the resolvents converge in the operator norm. This means that the spectral properties of $A_{\eta, 0} - m c^2$ and $A_{0, \tau} \mp m c^2$ are asymptotically the same for large $c$ as those of $-\Delta_\eta$ and $\mp\Delta_\tau$, respectively. Since the spectral properties of Schr\"odinger operators with $\delta$-potentials are well-studied, see, e.g., the review \cite{E08} or the monograph \cite{EK15}, one can deduce many effects for the corresponding Dirac operators as well. As an example of this idea the following lemma is shown in \cite[Proposition~5.5]{BEHL17}; a similar statement can also be proved for $A_{0, \tau}$.

\begin{lem}
  Let $j \in \mathbb{N}$. Then there is an $\eta < 0$ sufficiently large such that the number of eigenvalues of $A_{\eta, 0}$ in the gap $(-m c^2, m c^2)$ of $\sigma_{\textup{ess}}(A_{\eta, 0})$
is larger than~$j$ for all sufficiently large $c$. 
\end{lem}

In the rest of this section we sketch how Theorem~\ref{theorem_nonrelativistic_limit} can be shown; for details on the proof for the statement on $A_{\eta, 0}$ see \cite[Section~5]{BEHL17}, the claim for $A_{0, \tau}$ can be verified with the same arguments. We also only discuss the convergence of $A_{\eta, 0}$ here.

\begin{proof}[Sketch of the proof of Theorem~\ref{theorem_nonrelativistic_limit}]
Having the Krein type resolvent formula from Theorem~\ref{theorem_self_adjoint} in mind, one expects that it suffices to investigate the limiting behavior of $(A_0 - (\lambda + m c^2))^{-1}$, $\Phi_{\lambda + m c^2}$,
$\mathcal{C}_{\lambda + m c^2}$ and $\Phi_{\bar{\lambda} + m c^2}^*$.
For that, we state first a similar resolvent formula for $-\Delta_\eta$.
We define for $\lambda \in \mathbb{C} \setminus \mathbb{R}$ the function
\begin{equation*} 
  K_\lambda(x) := 2 m \frac{e^{i \sqrt{2 m \lambda} |x|}}{4 \pi |x|}, \qquad x \in \mathbb{R}^3 \setminus \{ 0 \},
\end{equation*}
and recall that
\begin{equation*} 
  \left( -\frac{1}{2 m} \Delta - \lambda \right)^{-1} f(x) = \int_{\mathbb{R}^3} K_\lambda(x-y) f(y) \text{d} y,
  \qquad f \in L^2(\mathbb{R}^3),~ x \in \mathbb{R}^3,
\end{equation*}
see for instance \cite[Chapter~7.4]{T14}.
Moreover, we introduce the bounded integral operators
$\Psi_\lambda: L^2(\Sigma) \rightarrow L^2(\mathbb{R}^3)$ acting as
\begin{equation*} 
  \Psi_\lambda \varphi(x) := \int_\Sigma K_\lambda(x-y) \varphi(y) \text{d} \sigma(y),
  \qquad \varphi \in L^2(\Sigma), ~x \in \mathbb{R}^3,
\end{equation*}
and $\mathcal{D}_\lambda: L^2(\Sigma) \rightarrow L^2(\Sigma)$,
\begin{equation*}
  \mathcal{D}_\lambda \varphi(x) := \int_\Sigma K_\lambda(x-y) \varphi(y) \text{d} \sigma(y),
  \qquad \varphi \in L^2(\Sigma),~ x \in \Sigma.
\end{equation*}
Furthermore, a simple calculation shows that the adjoint 
$\Psi_\lambda^*: L^2(\mathbb{R}^3) \rightarrow L^2(\Sigma)$ is 
\begin{equation*}
  \Psi_\lambda^* f(x) = \int_{\mathbb{R}^3} K_{\bar{\lambda}}(x-y) f(y) \text{d} y,
  \qquad f \in L^2(\mathbb{R}^3), ~x \in \Sigma.
\end{equation*}
Then it is verified, e.g., in \cite[Theorem~3.5]{BLL13} or \cite[Lemma~2.3]{BEKS94} that
for all $\lambda \in \mathbb{C} \setminus \mathbb{R}$ the operator $I_1 + \eta \mathcal{D}_\lambda$ is boundedly invertible in $L^2(\Sigma)$ and 
\begin{equation} \label{resolvent_Schroedinger}
  (-\Delta_\eta - \lambda)^{-1} = \left( -\frac{1}{2 m} \Delta - \lambda \right)^{-1}
      - \Psi_\lambda (I_1 + \eta \mathcal{D}_\lambda)^{-1} \eta \Psi_{\bar{\lambda}}^*.
\end{equation}
Now concerning the limiting behavior of $(A_0 - (\lambda + m c^2))^{-1}$, $\Phi_{\lambda + m c^2}$,
$\mathcal{C}_{\lambda + m c^2}$ and $\Phi_{\bar{\lambda} + m c^2}^*$ it is proven in \cite[Proposition~5.2]{BEHL17} that there exists for any $\lambda \in \mathbb{C} \setminus \mathbb{R}$ a constant $K > 0$ independent of $c$ such that 
\begin{equation*}
  \begin{split}
  \left\| \big(A_0 - (\lambda + m c^2)\big)^{-1} - \left( -\frac{1}{2 m} \Delta - \lambda \right)^{-1} P_+ \right\| &
      \leq \frac{K}{c}; \\
  \| \Phi_{\lambda + m c^2} - \Psi_{\lambda} P_+ \| \leq \frac{K}{c}; \qquad
  \| \mathcal{C}_{\lambda + m c^2} - \mathcal{D}_{\lambda}& P_+ \| \leq \frac{K}{c}.
  \end{split}
\end{equation*}
Combining this with the resolvent formula for $A_{\eta, 0}$ from Theorem~\ref{theorem_self_adjoint} and~\eqref{resolvent_Schroedinger} one deduces the claim of Theorem~\ref{theorem_nonrelativistic_limit}.
\end{proof}



\bibliographystyle{alpha}

\newcommand{\etalchar}[1]{$^{#1}$}

\end{document}